\documentclass[12pt,titlepage]{amsart}
\usepackage{amsopn}
\usepackage{amssymb}
\usepackage{amsaddr}
\usepackage{graphicx}
\usepackage[all]{xy}
\usepackage{amscd}
\usepackage{color}
\usepackage[colorlinks]{hyperref}
\usepackage{amsmath,amsthm,amsfonts,amssymb}
\usepackage{verbatim}
\usepackage{latexsym}

\newtheorem{thm}{Theorem}[section]

 \newtheorem{prop}[thm]{Proposition}
 \newtheorem{rem}[thm]{Remark}
 \newtheorem{defn}[thm]{Definition}
 \newtheorem{exam}[thm]{Example}
 
 \theoremstyle{definition}
 \theoremstyle{remark}

\begin{document}
\title[AR-frames]
{Adaptive Resolution Frames: A Multilevel Framework in Hilbert
Spaces }

\author{Jahangir Cheshmavar}
\author[J. Cheshmavar]{}

\address{Jahangir Cheshmavar, Department of Mathematics,
Payame Noor University, P.O.Box 19395-4697, Tehran, Iran \\
\emph{j$_{_-}$cheshmavar@pnu.ac.ir}\\
}
 \dedicatory{} \commby{}
\begin{abstract}
We introduce adaptive resolution frames ($AR$-frames), a class of
weighted frames equipped with a partitioned index structure that
models multiple resolution levels in Hilbert spaces. We establish
a block-triangular representation of the associated operators and
prove the convergence of an iterative reconstruction method with
explicit error estimates. We further derive optimal block-diagonal
preconditioners that minimize the iteration error and obtain
stability results showing that $AR$-frames and their canonical
reconstruction operators remain stable under sufficiently small
weighted perturbations. These results provide a unified framework
for efficient reconstruction and perturbation analysis of
$AR$-frames.
\end{abstract}
\maketitle

\section{\textbf{Introduction and Preliminaries}}\label{Sec1}
Frames in Hilbert spaces play a fundamental role in modern
harmonic analysis, signal processing, sampling theory, and data
representation due to their ability to provide stable and
redundant expansions of vectors (see, e.g.,
\cite{Benedetto.1995,Christensen.2023,Eldar.2012,Heil.2011}).
Since the pioneering work of Duffin and Schaeffer
\cite{Duffin.1952}, frame theory has evolved into a powerful
mathematical framework with applications ranging from wavelet
analysis and image compression (see, e.g.,
\cite{Benedetto.1998,Meyer.1997}) to wireless communications and
machine learning \cite{Rajupillai.2019,Strohmer.2001}. Unlike
orthonormal bases, frames allow redundancy, which offers
robustness against noise, perturbations, and data loss while
preserving stable reconstruction properties. In many practical
applications, however, information is naturally organized across
multiple scales or resolution levels. Examples arise in weighted
or partitioned frames, multi-resolution image analysis,
multi-scale signal decomposition, adaptive sensing, and
coarse-to-fine numerical approximations, (see, e.g.,
\cite{Balazs.2010,Balazs.2026,Cai.2012,Mallat.1999}). Classical
frame theory does not explicitly encode this layered structure.
Motivated by these considerations, we introduce and study a new
class of frames, called adaptive resolution frames ($AR$-frames),
which incorporate both weighted coefficients and a partitioned
index structure representing different resolution levels. An
$AR$-frame consists of a family of vectors indexed by a countable
set decomposed into subsets corresponding to distinct scales.
Positive weights are assigned to the frame elements, allowing
different resolution levels to contribute adaptively to the
representation of signals, and the partition structure permits
coarse and fine components of a signal to be analyzed separately,
making $AR$-frames particularly suitable for adaptive
approximation and denoising problems (Definition \ref{def.1}). The
multi-resolution structure allows reconstruction formulas to be
performed level by level through truncated operators, yielding
approximations that converge in norm to the original signal.
Despite their usefulness, two central challenges remain in the
theory of $AR$-frames:
\begin{itemize}
\item[(i)] developing stable level-wise reconstruction methods
under multi-scale decompositions and inter-level operator
interactions. \item[(ii)] ensuring stability under perturbation,
since small changes in the frame elements or weights may affect
the frame property and reconstruction operators, thereby
compromising reliability in practical applications.
\end{itemize}
We first establish a block decomposition of the associated frame
operators and derive a Neumann-type iterative reconstruction
scheme with geometric convergence (Proposition \ref{Prop.1}). We
next present the stability of adaptive level-wise preconditioning
and derives the optimal scaling that minimizes the worst-case
reconstruction error. (Proposition \ref{Prop.22}). Another central
aspect of this work is the stability analysis of $AR$-frames. We
prove that sufficiently small perturbations of the frame vectors
preserve the $AR$-frame property and derive explicit estimates for
the perturbed frame bounds (Proposition \ref{Prop.2}). In
addition, we investigate perturbations at the level of weighted
analysis operators and show that the associated reconstruction
operators admit convergent Neumann-series representations
(Proposition \ref{Prop.3}). These results guarantee the robustness
of $AR$-frame reconstructions under noise and approximation
errors, which is essential for practical applications. To
demonstrate the practical relevance of the theory, we present
finite-dimensional examples illustrating adaptive reconstruction
and denoising. These examples highlight the usefulness of
$AR$-frames in signal recovery and multi-scale
data representation.\\

\textbf{Motivation and Novelty}

Classical frames provide stable expansions but lack flexibility
for multi-resolution or weighted representations, which are
crucial in applications such as signal denoising, image
compression, and adaptive feature extraction. $AR$-frames address
this by allowing level-wise partitions and adaptive weights,
enabling more accurate, scale-sensitive approximations while
handling structured data effectively. $AR$-frames combine
weighted, multi-resolution representations with robust
perturbation stability. Compared to classical frames, $g$-frames,
and scalable frames, they enable adaptive level-wise
reconstruction, error control, and operator-compatible
decompositions, making them more flexible and computationally
efficient for signal and data analysis.

Throughout the paper, $\mathcal{H}$ is a separable Hilbert space
with inner product $\langle \cdot, \cdot \rangle$ and norm $\|
\cdot \|$, $\mathcal{B}(\mathcal{H})$ is the collection of all
bounded linear operators on $\mathcal{H}$, and $I$ is a countable
index set, possibly partitioned as $I = \bigcup_{k=1}^\infty I_k$.

\begin{defn} \label{def.1}
The family $\{ f_i \}_{i \in I} \subset \mathcal{H}$, equipped
with positive weights $\{ \lambda_i \}_{i \in I}$ and a partition
of the index set $I = \bigcup_{k=1}^\infty I_k$ with $I_k \cap I_j
= \emptyset \text{ for } k \neq j$ is an adaptive resolution frame
(AR-frame) for $\mathcal{H}$ if, there exist constants $A, B> 0$
such that
\begin{eqnarray}\label{form1}
 A \|x\|^2 \leq \sum_{k=1}^\infty
\sum_{i \in I_k} \lambda_i | \langle x, f_i \rangle |^2 \leq B
\|x\|^2, \,\ \forall x \in \mathcal{H}.
 \end{eqnarray}
The constants $A$ and $B$ are called the frame bounds.
\end{defn}
An $AR$-frame $\{ f_i \}_{i \in I} \subset \mathcal{H}$ with
weights $\{ \lambda_i \}_{i \in I}$ and index partition
$\{I_k\}_{k=1}^\infty$ is tight if
\begin{eqnarray*}
\sum_{k=1}^\infty \sum_{i \in I_k} \lambda_i | \langle x, f_i
\rangle |^2 = A \|x\|^2, \quad \forall x \in \mathcal{H},
\end{eqnarray*}
for some constant $A > 0$. It is a Parseval $AR$-frame if $A=1$.
\begin{rem}
(i) each subset $I_k$ can be thought of as corresponding to a
level of resolution or scale, multi-scale structure through the
partition $\{ I_k \}$, allowing different groups of vectors to
operate at different resolutions. (ii) If all $\lambda_i = 1$ and
set just one resolution level $I_k = I$, then an AR-frame reduces
to a classical frame in $\mathcal{H}$.
\end{rem}

Develop now the concept of dual $AR$-frames, which allow for
reconstruction of elements in $\mathcal{H}$. Let $\{f_i\}_{i\in
I}\subset\mathcal H$ be an $AR$-frame with orthogonal
decomposition
\begin{eqnarray*}
\mathcal H=\bigoplus_{k=1}^{\infty}\mathcal H_k, \qquad \mathcal
H_k=\overline{\operatorname{span}}\{f_i:i\in I_k\},
\end{eqnarray*}
where each $\{f_i\}_{i\in I_k}$ is a frame for $\mathcal H_k$.
Define the level-wise frame operators
\begin{eqnarray*}
S_k:\mathcal{H}_k \rightarrow \mathcal{H}_k, \,\ S_kx = \sum_{i
\in I_k} \lambda_i \langle x, f_i \rangle f_i, \,\ \forall x \in
\mathcal{H}_k,
\end{eqnarray*}
and the full $AR$-frame operator $S:=\sum_{k=1}^{\infty}S_k$,
where the sum converges in the strong operator topology. The
operator $S$ is bounded, self-adjoint, positive, and invertible.
This follows from the $AR$-frame inequality and the fact that all
$\lambda_i > 0$. For every $N \in \mathbb{N}$, define
\begin{eqnarray*}
S^{(N)} := \sum_{k=1}^N S_k,
\end{eqnarray*}
then $0 \leq S^{(N)} \leq S$, and the sequence $\{S^{(N)}\}_{N \in
\mathbb{N}}$ is an increasing sequence of positive operators
strongly converging to $S$. It is note that, there is no guarantee
that $S^{(N)}$ is invertible on all of $\mathcal H$. In general,
the first $N$ levels may fail to span $\mathcal H$, it is
invertible only on
\begin{eqnarray*}
\mathcal H^{(N)} := \overline{\operatorname{span}} \{f_i:i\in
I_1\cup\cdots\cup I_N\}.
\end{eqnarray*}
The canonical dual $AR$-frame vectors $\tilde{f}_i :=
S^{-1}(\lambda_if_i)$ admit the decomposition as follows:
\begin{eqnarray}
\tilde{f}_i = \lim_{N \to \infty} (S^{(N)})^{-1} (\lambda_if_i)
\quad \text{for }
i \in I_N,
\end{eqnarray}
reflecting the multi-resolution structure and allowing adaptive
level-wise inversion. We can approximate reconstruction by
truncated levels, for each $x \in H$, the truncated partial
reconstructions
\begin{eqnarray}
x_N :=  \sum_{k=1}^N \sum_{i \in I_k} \lambda_i \langle x, f_i
\rangle (S^{(N)})^{-1}f_i ,
\end{eqnarray}
converge in norm to $x$ as $N \to \infty$.\\

Here are concrete simple example showing how an $AR$-frame can be
used to denoise a signal by truncating higher resolution levels.

\begin{exam} \label{exe.1}
Consider a signal $x \in \mathbb{R}^2$ with noise at the fine
resolution level. We use an $AR$-frame to reconstruct a denoised
approximation. Define a simple $AR$-frame in
$\mathbb{R}^2$, $f_1 = \begin{bmatrix}1 \\
0\end{bmatrix}, \,\ f_2 = \begin{bmatrix}0 \\ 1\end{bmatrix},
\mbox{and}\,\ f_3 = \frac{1}{\sqrt{2}} \begin{bmatrix}1 \\
1\end{bmatrix}$, with resolution weights $\lambda_1 = 1, \,\
\lambda_2 = 1$, and $\lambda_3 = 2$. Consider partition into
resolution levels, $I_1 = \{1, 2\}$ (coarse), and $I_2 = \{3\}$
(fine). Let the original signal be $x_{\text{orig.}} =
\begin{bmatrix}3 \\ 1\end{bmatrix}$. Suppose high-frequency noise
corrupts the fine component $f_3$. We add noise $n = \alpha \cdot
f_3$, with $\alpha = 5$. So the observed signal is
\begin{eqnarray*}
x = x_{\text{orig.}} + n = \begin{bmatrix}3 \\
1\end{bmatrix} + 5 \cdot \frac{1}{\sqrt{2}} \begin{bmatrix}1 \\
1\end{bmatrix} \approx \begin{bmatrix}6.54 \\
4.54\end{bmatrix}.
\end{eqnarray*}
By compute $Sx$ for $x =
\begin{bmatrix}x_1 \\ x_2\end{bmatrix}$, we have that
\begin{eqnarray*}
S x = \sum_{i=1}^3 \lambda_i \langle x,f_i\rangle f_i&=&\lambda_1
x_1 f_1 + \lambda_2 x_2 f_2 + \lambda_3 \left(
\frac{x_1 + x_2}{\sqrt{2}} \right) f_3\\
&=& x_1\begin{bmatrix}1 \\ 0\end{bmatrix} + x_2 \begin{bmatrix}0 \\
1\end{bmatrix} + 2 \cdot \left( \frac{x_1 + x_2}{\sqrt{2}} \right)
\cdot \frac{1}{\sqrt{2}} \begin{bmatrix}1 \\ 1\end{bmatrix}= \begin{bmatrix}2x_1 + x_2 \\
x_1 + 2x_2\end{bmatrix}.
\end{eqnarray*}
So,
\begin{eqnarray*}
S =\begin{bmatrix} 2 & 1 \\ 1 & 2 \end{bmatrix}, \,\ \mbox{and} \quad S^{-1} = \frac{1}{3} \begin{bmatrix} 2 & -1 \\
-1 & 2
\end{bmatrix}.
\end{eqnarray*}

Also, $\tilde{f}_i = S^{-1}(\lambda_i f_i)$, and then
\begin{eqnarray*}
 \tilde{f}_1 =
\frac{1}{3}
\begin{bmatrix}2 \\ -1\end{bmatrix}, \,\ \tilde{f}_2 =
\frac{1}{3} \begin{bmatrix}-1 \\ 2\end{bmatrix}, \,\  \tilde{f}_3
= \frac{2}{3} S^{-1} \left( \frac{1}{\sqrt{2}}
\begin{bmatrix}1 \\ 1\end{bmatrix} \right) = \frac{2}{3\sqrt{2}}
\begin{bmatrix}1 \\ 1\end{bmatrix}.
\end{eqnarray*}
For the analyze and reconstruct (coarse only) compute coarse
coefficients, we have $\langle x, f_1 \rangle = 6.54, \langle x,
f_2 \rangle = 4.54$ and ignore fine coefficient $\langle x, f_3
\rangle$. Then
\begin{eqnarray*}
 x^{(1)} &=& \langle x, f_1 \rangle \tilde{f}_1 +
\langle x, f_2 \rangle \tilde{f}_2 \\
&=& 6.54 \cdot \frac{1}{3}
\begin{bmatrix}2
\\ -1\end{bmatrix} + 4.54 \cdot \frac{1}{3} \begin{bmatrix}-1 \\
2\end{bmatrix} \approx \begin{bmatrix}2.85 \\
0.85\end{bmatrix}.
\end{eqnarray*}
But the original signal is $x_{\text{orig.}} =
\begin{bmatrix}3 \\ 1\end{bmatrix}$, and noisy signal is $x =
\begin{bmatrix}6.54 \\ 4.54\end{bmatrix}$. Denoised approximation
using only coarse levels, we have $x^{(1)} \approx
\begin{bmatrix}2.85 \\ 0.85\end{bmatrix}$. This denoised version is
close to the original, having filtered out the high-frequency
noise in $f_3$.
\end{exam}

\section{\textbf{The results}}
Throughout the section, let $\{f_i\}_{i\in I}\subset\mathcal H$ be
an $AR$-frame with frame bounds $A, B$, positive weights
$\{\lambda_i\}_{i\in I}$, partition $I=\bigcup_{k=1}^\infty I_k$,
and with orthogonal decomposition
\begin{eqnarray*}
\mathcal H=\bigoplus_{k=1}^{\infty}\mathcal H_k, \qquad \mathcal
H_k=\overline{\operatorname{span}}\{f_i:i\in I_k\},
\end{eqnarray*}
where each $\{f_i\}_{i\in I_k}$ is a frame for $\mathcal H_k$.
Also, let
\begin{eqnarray*}
S_k:\mathcal H_k\rightarrow\mathcal H_k, \,\ S_kx:=\sum_{i\in
I_k}\lambda_i\langle x,f_i\rangle f_i,
\end{eqnarray*}
for any $x\in\mathcal H_k$, be the level-wise frame operators, and
$S=\sum_{k=1}^\infty S_k$ be the full frame operator.\\

Our first result establishes a block operator representation for
adaptive resolution frames and provides a convergent iterative
reconstruction method with an explicit error estimate.

\begin{prop} \label{Prop.1}
Suppose that $T\in\mathcal B(\mathcal H)$ satisfies $T(\mathcal
H_k)\subseteq\mathcal H_{k+1}, \,\ k\ge1$, and
$\rho:=\|S^{-1}T\|<1$. Then the following statements hold:
\begin{itemize}
\item[(i)] There exists an orthonormal basis of $\mathcal H$ in
which every $S_k$ is diagonal, the frame operator $S$ is block
diagonal, and $T$ is strictly lower triangular with respect to the
decomposition $\mathcal H=\bigoplus_{k=1}^{\infty}\mathcal H_k$.

\item[(ii)] For every $x\in\mathcal H$, define the sequence
$\{u_n\}_{n\ge0}$ by $u_0=0$, and
\begin{eqnarray} \label{eqn-90}
u_{n+1} = u_n + S^{-1}\bigl(x-(S-T)u_n\bigr), \,\ n\ge0.
\end{eqnarray}
Then $\{u_n\}_{n\geq 0}$ converges in norm to the unique solution
$u=(S-T)^{-1}x$. Moreover, $\|u-u_n\| \le \rho^{\,n}\|u\|, \,\
n\ge0$.
\end{itemize}
\end{prop}

\begin{proof}
Since $\mathcal H=\bigoplus_{k=1}^{\infty}\mathcal H_k$ is an
orthogonal decomposition, every $\mathcal H_k$ is a closed
subspace of $\mathcal H$. For each $k$, the operator
\begin{eqnarray*}
S_kx=\sum_{i\in I_k}\lambda_i\langle x,f_i\rangle f_i
\end{eqnarray*}
maps $\mathcal H_k$ into itself. Moreover, if $x\in\mathcal H_j$
with $j\neq k$, then $\langle x,f_i\rangle=0, \,\ i\in I_k$,
because $\mathcal H_j\perp\mathcal H_k$. Hence $S_kx=0$.
Therefore, $S=\sum_{k=1}^{\infty}S_k=\bigoplus_{k=1}^{\infty}S_k$,
so the frame operator is block diagonal with respect to the
decomposition $\mathcal H=\bigoplus_{k=1}^{\infty}\mathcal H_k$.
Since every $S_k$ is bounded, positive, self-adjoint and
invertible on $\mathcal H_k$, the spectral theorem provides an
orthonormal basis $\{e_i^{(k)}\}_{i\geq 1}$ of $\mathcal H_k$
consisting of eigenvectors of $S_k$,
\begin{eqnarray*}
S_ke_i^{(k)} = \sigma_i^{(k)}e_i^{(k)}, \qquad \sigma_i^{(k)}>0.
\end{eqnarray*}
Taking the union of these bases over all $k$ gives an orthonormal
basis of $\mathcal H$. Relative to this basis every $S_k$ is
diagonal, and consequently $S$ is block diagonal. Furthermore,
$$T(\mathcal H_k)\subseteq\mathcal H_{k+1},$$ implies that, relative
to the decomposition $\mathcal H=\bigoplus_{k=1}^{\infty}\mathcal
H_k$, the operator $T$ has the block representation
\begin{eqnarray*}
T=
\begin{bmatrix} 0&0&0&\cdots\\ T_1&0&0&\cdots\\ 0&T_2&0&\cdots\\
\vdots&\vdots&\vdots&\ddots \end{bmatrix},
\end{eqnarray*}
where, $T_k:=T|_{\mathcal H_k}:\mathcal H_k \rightarrow \mathcal
H_{k+1}$. Thus $T$ is strictly lower triangular. This proves
$(i)$. To prove $(ii)$, define $A:=S-T$. Since
$\|S^{-1}T\|=\rho<1$, the Neumann series theorem implies that
$I-S^{-1}T$ is invertible. Because $A = S(I-S^{-1}T)$, and $S$ is
invertible as the frame operator of an $AR$-frame, it follows that
$A=S-T$ is invertible. The iteration (\ref{eqn-90}) can be
rewritten as
\begin{eqnarray} \label{eqn-1000} \notag
u_{n+1} &= u_n+S^{-1}(x-Au_n)\\ &= S^{-1}x+S^{-1}Tu_n.
\end{eqnarray}

Let $u=A^{-1}x=(S-T)^{-1}x$. Then $Au=Su-Tu=x$, or
\begin{eqnarray} \label{eqn-1001}
u=S^{-1}x+S^{-1}Tu.
\end{eqnarray}
Define the error $e_n=u-u_n$. Then two identities (\ref{eqn-1000})
and (\ref{eqn-1001}) gives
\begin{eqnarray*}
\begin{aligned} e_{n+1} &=
S^{-1}Tu-S^{-1}Tu_n\\ &= S^{-1}T(u-u_n)\\ &= S^{-1}Te_n.
\end{aligned}
\end{eqnarray*}
Hence, $e_n=(S^{-1}T)^ne_0$. Therefore,
\begin{eqnarray*}
\|e_n\| \le \|S^{-1}T\|^n\|e_0\| = \rho^n\|u\|,
\end{eqnarray*}
because $e_0=u-u_0=u$. Since $0<\rho<1$, we have
$\rho^n\longrightarrow0$, and consequently
\begin{eqnarray*}
u_n\longrightarrow u=(S-T)^{-1}x,
\end{eqnarray*}
in norm. This completes the proof.
\end{proof}
\begin{rem}
This proposition uses $T$ in an essential way. It is an
application of the Richardson (or Picard) iteration
\cite{Richard,Yousef} to the block operator equation $(S-T)u=x$.
The hypothesis $\rho=\|S^{-1}T\|<1$ guarantees convergence by the
Neumann series theorem, while the block structure induced by the
$AR$-frame gives the adaptive multilevel interpretation of the
iteration.
\end{rem}

Our second result establishes the connection between adaptive
multiresolution frame decompositions and level-wise
preconditioning, and identifies the optimal block-diagonal scaling
that minimizes the worst-case convergence factor of the iterative
reconstruction.

\begin{prop} \label{Prop.22}
Let $0<A_k I_{\mathcal H_k} \leq S_k \leq B_k I_{\mathcal H_k},
\,\ k\geq 1$, where $0<A_k\leq B_k<\infty $. Also, let
\begin{eqnarray} \label{eqn-91}
P_{\alpha} = \bigoplus_{k=1}^{\infty}\alpha_k S_k^{-1},
\end{eqnarray}
where $\alpha_k>0$ are scaling parameters satisfy
$\sup_{k\ge1}\frac{\alpha_k}{A_k}<\infty$. Then the following
statements hold:
\begin{itemize}
\item[(i)] $P_{\alpha}\in\mathcal B(\mathcal H)$, and $P_\alpha S
= \bigoplus_{k=1}^{\infty}\alpha_k I_{\mathcal H_k}$. If, in
addition $$0<\alpha_- \leq \alpha_k\leq \alpha_+<\infty ,$$ then
$\alpha_- I \leq P_\alpha S \leq \alpha_+ I$. \item[(ii)] Consider
the iterative reconstruction
\begin{eqnarray} \label{eqn-92}
x_{n+1} = x_n+ P_\alpha(y-Sx_n),
\end{eqnarray}
where $y=Sx$. Then the error operator is $E_\alpha = I-P_\alpha
S$, and the error satisfies $e_{n+1} = E_\alpha e_n$, where
$e_n=x-x_n$. Consequently,
$$\|x-x_n\| \leq \|E_\alpha\|^n\|x-x_0\|.$$
Among all block-diagonal preconditioners of the form
(\ref{eqn-91}), the optimal scalar normalization on each level is
\begin{eqnarray*}
\alpha_k^{*} = \frac{2} {\lambda_{\max}(S_k)+\lambda_{\min}(S_k)}=
\frac{2}{A_k+B_k},
\end{eqnarray*}
for which
\begin{eqnarray*}
\|I-\alpha_k^{*}S_k\| = \frac{B_k-A_k}{B_k+A_k},
\end{eqnarray*}
Consequently,
\begin{eqnarray*}
\|E_{\alpha^*}\| = \sup_{k\geq1} \frac{B_k-A_k}{B_k+A_k}.
\end{eqnarray*}
and this is the smallest possible operator norm among all
block-diagonal preconditioners of the above form.
\end{itemize}
\end{prop}
\begin{proof}
Since $\mathcal H=\bigoplus_{k=1}^{\infty}\mathcal H_k$, every
vector $x\in\mathcal H$ admits the unique orthogonal decomposition
$x=\sum_{k=1}^{\infty}x_k, \,\ x_k\in\mathcal H_k$, with $\|x\|^2
= \sum_{k=1}^{\infty}\|x_k\|^2$. Moreover, $$S=
\bigoplus_{k=1}^{\infty}S_k.$$ Since $A_kI_{\mathcal H_k} \le
S_k$, each $S_k$ is positive and invertible on $\mathcal H_k$.
Furthermore, $\|S_k^{-1}\| = \frac1{\lambda_{\min}(S_k)} \le
\frac1{A_k}$. Hence, for every $x=\sum_kx_k, \,\ P_\alpha x =
\sum_{k=1}^{\infty} \alpha_kS_k^{-1}x_k$, and therefore
\begin{eqnarray*}
\|P_\alpha x\|^2 &=& \sum_{k=1}^{\infty} \alpha_k^2
\|S_k^{-1}x_k\|^2 \\ &\le& \sum_{k=1}^{\infty}
\frac{\alpha_k^2}{A_k^2} \|x_k\|^2 \\ &\le& \left( \sup_{k\ge1}
\frac{\alpha_k}{A_k} \right)^2 \sum_{k=1}^{\infty}\|x_k\|^2.
\end{eqnarray*}
Thus $\|P_\alpha x\| \le \left( \sup_{k\ge1} \frac{\alpha_k}{A_k}
\right)\|x\|$, showing that $P_\alpha\in\mathcal B(\mathcal H)$.
Next,
\begin{eqnarray} \label{eqn-93}
P_\alpha S = \left( \bigoplus_{k=1}^{\infty} \alpha_kS_k^{-1}
\right) \left( \bigoplus_{k=1}^{\infty}S_k \right).
\end{eqnarray}
Since products of block-diagonal operators are computed blockwise,
(\ref{eqn-93}) gives
\begin{eqnarray*}
P_\alpha S = \bigoplus_{k=1}^{\infty}
\alpha_kS_k^{-1}S_k=\bigoplus_{k=1}^{\infty} \alpha_kI_{\mathcal
H_k}.
\end{eqnarray*}
If $0<\alpha_-\le\alpha_k\le\alpha_+$, then for every
$x=\sum_kx_k$,

\begin{eqnarray*}
\langle P_\alpha Sx,x\rangle &= \sum_{k=1}^{\infty}
\alpha_k\|x_k\|^2.
\end{eqnarray*}
Therefore, $\alpha_-\|x\|^2 \le \langle P_\alpha Sx,x\rangle \le
\alpha_+\|x\|^2$, which is equivalent to $$\alpha_-I \le P_\alpha
S \le \alpha_+I.$$ This proves $(i)$.\\
For part $(ii)$, let $e_n=x-x_n $. The iteration (\ref{eqn-92})
gives
\begin{eqnarray*}
e_{n+1} &=& x-x_n-P_\alpha(y-Sx_n)\\ &=& x-x_n-P_\alpha(Sx-Sx_n)\\
&=& (I-P_\alpha S)e_n .
\end{eqnarray*}
Hence, $e_n=(I-P_\alpha S)^ne_0$, and therefore, $\|e_n\| \leq
\|I-P_\alpha S\|^n \|e_0\|$. To determine the optimal
normalization, consider a fixed level $k$. Since $S_k$ is
self-adjoint and positive, the spectrum of $S_k$ satisfies
\begin{eqnarray*}
\sigma(S_k) \subset[A_k,B_k].
\end{eqnarray*}
The restriction of the error operator to $\mathcal H_k$ is
$I-\alpha_kS_k$ whose operator norm satisfies
\begin{eqnarray*}
\|I-\alpha_kS_k\| = \max_{\lambda\in[A_k,B_k]}
|1-\alpha_k\lambda|.
\end{eqnarray*}
The optimal value of $\alpha_k$ minimizes this maximum. Indeed,
since the function $$\lambda \longmapsto |1-\alpha_k\lambda|,$$ is
convex on $[A_k,B_k]$, a basic property of convex functions is
that their maximum over a compact interval is attained at an
endpoint, so its maximum over the interval is attained at one of
the endpoints. Hence
\begin{eqnarray*}
\max_{\lambda\in[A_k,B_k]}|1-\alpha_k\lambda| =
\max\{|1-\alpha_kA_k|,\;|1-\alpha_kB_k|\}.
\end{eqnarray*}
The minimax value is achieved when these two endpoint errors are
equal; otherwise, one can perturb $\alpha_k$ to decrease the
larger error without increasing the smaller one beyond it. Thus
$|1-\alpha_kA_k| = |1-\alpha_kB_k|$, which, since
$1/B_k<\alpha_k<1/A_k$, reduces to $$1-\alpha_kA_k =
\alpha_kB_k-1.$$ Indeed, since $A_k,B_k>0$, the optimal $\alpha_k$
lies between $\frac{1}{B_k}$ and $\frac{1}{A_k}$. Indeed, if
$\alpha_k<1/B_k$, then both $1-\alpha_k A_k$ and $1-\alpha_k B_k$
are positive, and increasing $\alpha_k$ decreases both, so this
cannot be optimal; if $\alpha_k>1/A_k$, then both are negative,
and decreasing $\alpha_k$ decreases the maximum. Thus the
minimizer satisfies $\frac1{B_k} < \alpha_k < \frac1{A_k}$, so
that $1-\alpha_k A_k>0, \,\ 1-\alpha_k B_k<0$. Hence

$$|1-\alpha_k A_k| = 1-\alpha_k A_k, \,\ \mbox{and} \,\ |1-\alpha_k B_k| =
\alpha_k B_k-1.$$ Therefore the equality of endpoint errors
becomes $1-\alpha_k A_k = \alpha_k B_k-1$. Solving for $\alpha_k$
gives
\begin{eqnarray*}
\alpha_k^* = \frac{2}{A_k+B_k}.
\end{eqnarray*}
 Substituting this value, gives
\begin{eqnarray*}
\|I-\alpha_k^*S_k\| &=& 1-\alpha_k^*A_k\\ &=&
1-\frac{2A_k}{A_k+B_k}\\ &=& \frac{B_k-A_k}{B_k+A_k}.
\end{eqnarray*}
Finally,
\begin{eqnarray*}
I-P_{\alpha^*}S = \bigoplus_{k=1}^{\infty} (I-\alpha_k^*S_k),
\end{eqnarray*}
Since the norm of a block-diagonal operator acting on an
orthogonal direct sum equals the supremum of the norms of its
diagonal blocks,
 $\|I-P_{\alpha^*}S\| = \sup_k \|I-\alpha_k^*S_k\|$.
Hence,
\begin{eqnarray*}
\|I-P_{\alpha^*}S\| = \sup_k \|I-\alpha_k^*S_k\| = \sup_k
\frac{B_k-A_k}{B_k+A_k}.
\end{eqnarray*}
Because the minimization problem is independent on each orthogonal
level, no other choice of positive scalars $\{\alpha_k\}$ can
produce a smaller block norm. Hence $P_{\alpha^*}$ is the optimal
block-diagonal preconditioner within the class
\begin{eqnarray*}
\left\{ \bigoplus_{k=1}^{\infty}\alpha_kS_k^{-1}: \alpha_k>0,\,
\sup_k\frac{\alpha_k}{A_k}<\infty \right\},
\end{eqnarray*}
and the proof is complete.

\end{proof}

\begin{rem}
The proposition shows that an $AR$-frame naturally admits an
optimal level-wise preconditioner of the form
\begin{eqnarray*}
P_{\alpha^*} = \bigoplus_{k=1}^{\infty} \frac{2}{A_k+B_k}S_k^{-1}
\end{eqnarray*}
which adapts the reconstruction to the local conditioning of each
resolution level and minimize the worse-case contraction factor
among all such block-diagonal preconditioners.
\end{rem}

We now examine the stability of $AR$-frames, showing that small
perturbations of the frame vectors preserve the $AR$-frame
property and provide explicit bounds on the perturbed frame.
\begin{prop} \label{Prop.2}
Suppose that $\{g_i\}_{i\in I}\subset\mathcal H$
satisfies
\begin{eqnarray*}
\sum_{k=1}^{\infty}\sum_{i\in I_k} \lambda_i\|f_i-g_i\|^2 \le
\delta^2,
\end{eqnarray*}
where $0<\delta<\sqrt A$. Then $\{g_i\}_{i\in I}$, equipped with
the same weights $\{\lambda_i\}_{i\in I}$ and the same partition
$I=\bigcup_{k=1}^{\infty}I_k$, is also an $AR$-frame for $\mathcal
H$. Moreover, its frame bounds are $(\sqrt A-\delta)^2
\quad\text{and}\quad (\sqrt B+\delta)^2$.
\end{prop}
\begin{proof}
Define the weighted analysis operators
\begin{eqnarray*}
T_f,T_g:\mathcal H\longrightarrow \ell^2(I)
\end{eqnarray*}
by
\begin{eqnarray*}
T_fx = \{\sqrt{\lambda_i}\langle x,f_i\rangle\}_{i\in I}, \qquad
T_gx = \{\sqrt{\lambda_i}\langle x,g_i\rangle\}_{i\in I}.
\end{eqnarray*}
Since $\{f_i\}_{i\in I}$ is an $AR$-frame, we have
\begin{eqnarray*}
A\|x\|^2 \le \sum_{k=1}^{\infty}\sum_{i\in I_k} \lambda_i |\langle
x,f_i\rangle|^2 \le B\|x\|^2, \qquad x\in\mathcal H.
\end{eqnarray*}
Equivalently,
\begin{eqnarray*}
\sqrt A\,\|x\| \le \|T_fx\| \le \sqrt B\,\|x\|, \qquad
x\in\mathcal H.
\end{eqnarray*}
Next, for every $x\in\mathcal H$, we have
\begin{eqnarray*}
\|(T_f-T_g)x\|^2 &= \sum_{k=1}^{\infty}\sum_{i\in I_k} \lambda_i
\big| \langle x,f_i-g_i\rangle \big|^2.
\end{eqnarray*}
By the Cauchy-Schwarz inequality, we have
\begin{eqnarray*}
\|(T_f-T_g)x\|^2 &\le& \|x\|^2 \sum_{k=1}^{\infty}\sum_{i\in I_k}
\lambda_i \|f_i-g_i\|^2\\ &\le& \delta^2\|x\|^2.
\end{eqnarray*}
Hence, $\|T_f-T_g\| \le \delta$. Using the reverse triangle
inequality, we have
\begin{eqnarray*}
\|T_gx\| &\ge& \|T_fx\| - \|(T_f-T_g)x\|\\ &\ge& (\sqrt
A-\delta)\|x\|.
\end{eqnarray*}
Since $\delta<\sqrt A$, the quantity $\sqrt A-\delta$ is positive.
Therefore,

\begin{eqnarray} \label{eqn-95} \notag
\sum_{k=1}^{\infty}\sum_{i\in I_k} \lambda_i |\langle
x,g_i\rangle|^2 &=& \|T_gx\|^2\\ &\ge& (\sqrt A-\delta)^2\|x\|^2.
\end{eqnarray}
Similarly,

\begin{eqnarray*}
\|T_gx\| &\le& \|T_fx\| + \|(T_f-T_g)x\|\\ &\le& (\sqrt
B+\delta)\|x\|.
\end{eqnarray*}

Hence
\begin{eqnarray} \label{eqn-96} \notag
\sum_{k=1}^{\infty}\sum_{i\in I_k} \lambda_i |\langle
x,g_i\rangle|^2 &=& \|T_gx\|^2\\ &\le& (\sqrt B+\delta)^2\|x\|^2.
\end{eqnarray}

Combining the estimates (\ref{eqn-95}) and (\ref{eqn-96}) yields
\begin{eqnarray*}
(\sqrt A-\delta)^2\|x\|^2 \le \sum_{k=1}^{\infty}\sum_{i\in I_k}
\lambda_i |\langle x,g_i\rangle|^2 \le (\sqrt B+\delta)^2\|x\|^2,
\qquad x\in\mathcal H.
\end{eqnarray*}
Therefore, $\{g_i\}_{i\in I}$, equipped with the same weights and
the same partition, is an $AR$-frame for $\mathcal H$, with frame
bounds $(\sqrt A-\delta)^2 \quad\text{and}\quad (\sqrt
B+\delta)^2$. This completes the proof.

\end{proof}

The following result provides a convergent Neumann series
representation for the perturbed reconstruction operators,
yielding practical formulas for applications.

\begin{prop} \label{Prop.3}
Let the weighted analysis operator $T_1: \mathcal{H} \to
\ell^2(I)$ be defined by
\begin{eqnarray*}
T_1(x) = (\sqrt{\lambda_i} \langle x, f_i \rangle)_{i \in I},
\end{eqnarray*}
and suppose $\{ g_i \}_{i \in I} \subset \mathcal{H}$ is another
sequence with the same weights and partition, with weighted
analysis operator
\begin{eqnarray*}
T_2(x) = (\sqrt{\lambda_i} \langle x, g_i \rangle)_{i \in I}.
\end{eqnarray*}
Define the perturbation operator $E := T_1 - T_2$.
\begin{itemize}
\item[(i)] If $\|E\| < \sqrt{A}$, then $\{ g_i \}_{i \in I}$ is
also an $AR$-frame with the same weights and partition, and its
frame bounds are $(\sqrt{A} - \|E\|)^2$ and $(\sqrt{B} +
\|E\|)^2$. \item[(ii)] If, in addition, the perturbation satisfies
$\|E\| < \sqrt{B+A}-\sqrt{B}$, then the canonical reconstruction
operator $R_g = S_g^{-1} T_2^*$ can be expressed as a convergent
Neumann series perturbation of $R_f = S_f^{-1} T_1^*$, i.e.,
\begin{eqnarray*}
R_g = \sum_{n=0}^{\infty} (- S_f^{-1} \Delta S)^n \big(R_f -
S_f^{-1} E^* \big),
\end{eqnarray*}
where $S_f = T_1^* T_1$ and $S_g = T_2^* T_2$ are $AR$-frame
operators for $\{ f_i \}_{i \in I}$ and $\{ g_i \}_{i \in I}$,
respectively, and $\Delta S := S_g - S_f$.
\end{itemize}
\end{prop}
\begin{proof}
$(i)$ For any $x \in \mathcal{H}$, we have
\begin{eqnarray*}
\|T_2 x\| = \|T_1 x - E x\| \le \|T_1 x\| + \|E x\| \le \sqrt{B}
\|x\| + \|E\| \|x\| = (\sqrt{B} + \|E\|) \|x\|.
\end{eqnarray*}
Similarly, using the reverse triangle inequality, we have
\begin{eqnarray*}
\|T_2 x\| = \|T_1 x - E x\| \ge \|T_1 x\| - \|E x\| \ge \sqrt{A}
\|x\| - \|E\| \|x\| = (\sqrt{A} - \|E\|) \|x\|.
\end{eqnarray*}
Squaring both sides, we have
\begin{eqnarray*}
(\sqrt{A} - \|E\|)^2 \|x\|^2 \le \|T_2 x\|^2 = \sum_{k=1}^\infty
\sum_{i \in I_k} \lambda_i |\langle x, g_i \rangle|^2 \le
(\sqrt{B} + \|E\|)^2 \|x\|^2.
\end{eqnarray*}
Since $\|E\| < \sqrt{A}$, the lower bound is strictly positive.
Therefore, $\{ g_i \}$ is an $AR$-frame with the same weights
$\lambda_i$ and partition $\{ I_k \}$.\\
$(ii)$ For any $x \in \mathcal{H}$ we have
\begin{eqnarray*}
\langle S_g x, x \rangle = \|T_2 x\|^2 \ge (\sqrt{A} - \|E\|)^2
\|x\|^2 > 0, \,\ \forall x \neq 0.
\end{eqnarray*}
This shows $S_g$ is bounded, self-adjoint and strictly positive.
Thus, $S_g^{-1}$ exists and is bounded. The canonical
reconstruction operators are
\begin{eqnarray*}
R_f = S_f^{-1} T_1^*, \quad R_g = S_g^{-1} T_2^*.
\end{eqnarray*}
Now express $S_g$ as a perturbation of $S_f$ as follows:
\begin{eqnarray*}
S_g = S_f + \Delta S, \quad \Delta S := S_g - S_f = -T_1^* E - E^*
T_1 + E^* E.
\end{eqnarray*}
So
\begin{eqnarray*}
R_g = S_g^{-1} T_2^* = (S_f + \Delta S)^{-1} (T_1^* - E^*).
\end{eqnarray*}
Multiply both sides by $S_f + \Delta S$, we have
\begin{eqnarray*}
(S_f + \Delta S) R_g = T_1^* - E^* \quad \implies \quad S_f R_g =
T_1^* - E^* - \Delta S R_g.
\end{eqnarray*}
Hence,
\begin{eqnarray*}
R_g &=& S_f^{-1} T_1^* - S_f^{-1} E^* - S_f^{-1} \Delta S R_g\\
&=& R_f - S_f^{-1} E^* - S_f^{-1} \Delta S R_g,
\end{eqnarray*}
that is,
\begin{eqnarray*}
(I + S_f^{-1} \Delta S) R_g = R_f - S_f^{-1} E^*.
\end{eqnarray*}
Therefore,
\begin{eqnarray} \label{eqn-200}
 R_g = (I + S_f^{-1} \Delta S)^{-1} (R_f - S_f^{-1} E^*),
\end{eqnarray}
and the Neumann series should be applied to $(I + S_f^{-1} \Delta
S)^{-1}$. If $\|S_f^{-1}\Delta S\| < 1$, then $I + S_f^{-1} \Delta
S$ is invertible and its inverse can be written as a Neumann
series
\begin{eqnarray} \label{eqn-300}
(I + S_f^{-1} \Delta S)^{-1}=\sum_{n=0}^{\infty} (-S_f^{-1} \Delta
S)^{n}.
\end{eqnarray}
For this, we have that
\begin{eqnarray*}
\|S_f^{-1} \Delta S\| \le \|S_f^{-1} T_1^* E\| + \|S_f^{-1} E^*
T_1\| + \|S_f^{-1} E^* E\|.
\end{eqnarray*}
Using $\|S_f^{-1}\| \le \frac{1}{A}$ and $\|T_1\| = \|T_1^*\| \le
\sqrt{B}$, we have
\begin{eqnarray*}
\|S_f^{-1} \Delta S\| \le \frac{\sqrt{B}}{A} \|E\| +
\frac{\sqrt{B}}{A} \|E\| + \frac{1}{A} \|E\|^2 = \frac{2
\sqrt{B}}{A} \|E\| + \frac{1}{A} \|E\|^2.
\end{eqnarray*}
The Neumann series (\ref{eqn-300}) converges if $\|S_f^{-1} \Delta
S\| < 1$, which holds when $\|E\| < \sqrt{B + A} - \sqrt{B}$.
Thus,
\begin{eqnarray*}
(I + S_f^{-1} \Delta S)^{-1} = \sum_{n=0}^{\infty} (-S_f^{-1}
\Delta S)^n,
\end{eqnarray*}
and therefore, by (\ref{eqn-200}) we have
\begin{eqnarray*}
R_g = \sum_{n=0}^{\infty} (-S_f^{-1} \Delta S)^n (R_f - S_f^{-1}
E^*),
\end{eqnarray*}
converges in operator norm.
\end{proof}

\begin{exam}
Consider an AR-frame in $\mathbb{R}^2$ as in Example
(\ref{exe.1}). The canonical reconstruction operator is
\begin{eqnarray*}
R_f = S_f^{-1} T_1^*, \quad \text{where } T_1^* c = \sum_i
\sqrt{\lambda_i} c_i f_i.
\end{eqnarray*}
Let's perturb $f_1$ slightly,
\begin{eqnarray*}
g_1 = f_1 + \epsilon
\begin{bmatrix} 0 \\ 0.1 \end{bmatrix} = \begin{bmatrix} 1 \\
0.1\end{bmatrix}, \quad g_2 = f_2, \quad g_3 = f_3
\end{eqnarray*}
with the same weights and partition. The weighted analysis
operators are
\begin{eqnarray*}
T_1(x) =
\begin{bmatrix} \langle x,f_1\rangle \\ \langle x,f_2\rangle \\
\sqrt{2}\langle x,f_3\rangle \end{bmatrix}, \quad T_2(x) =
\begin{bmatrix} \langle x,g_1\rangle \\ \langle x,g_2\rangle \\
\sqrt{2}\langle x,g_3\rangle
\end{bmatrix}.
\end{eqnarray*}
The perturbation operator is $E = T_1 - T_2 $, and then,
\begin{eqnarray*}
E(x) =
\begin{bmatrix} \langle x,f_1 - g_1\rangle \\ 0 \\ 0 \end{bmatrix}
= \begin{bmatrix} -0.1 x_2 \\ 0 \\ 0
\end{bmatrix},
\end{eqnarray*}
and its operator norm is $\|E\| = 0.1$, which is smaller than
$\sqrt{A} = 1$ (lower $AR$-frame bound of coarse level), so the
proposition (\ref{Prop.3}) applies. $S_g = T_2^* T_2 = S_f +
\Delta S$ with
\begin{eqnarray*}
\Delta S = S_g - S_f = - T_1^* E - E^* T_1 + E^* E.
\end{eqnarray*}
This is a small $2 \times 2$ perturbation, because only $f_1$
changed slightly. We write
\begin{eqnarray*}
R_g = S_g^{-1} T_2^* = (S_f + \Delta S)^{-1} (T_1^* - E^*) = (I +
S_f^{-1} \Delta S)^{-1} (R_f - S_f^{-1} E^*).
\end{eqnarray*}
Since $\|S_f^{-1} \Delta S\|$ is small $(<1)$, then
\begin{eqnarray*}
R_g = \sum_{n=0}^{\infty} (-S_f^{-1} \Delta S)^n (R_f - S_f^{-1}
E^*).
\end{eqnarray*}
Let's reconstruct a signal $x = \begin{bmatrix} 1 \\ 2
\end{bmatrix}$, using only first-order Neumann approximation $(n=0)$,
\begin{eqnarray*}
R_g x \approx R_f x - S_f^{-1} E^* x.
\end{eqnarray*}
By a simple computation, we have
\begin{eqnarray*}
T_1^* x = \sum_{i=1}^3 \sqrt{\lambda_i} \langle x, f_i \rangle f_i
\approx \begin{bmatrix}3.1213 \\
4.1213\end{bmatrix},
\end{eqnarray*}
\begin{eqnarray*}
R_f x = S_f^{-1} T_1^* x =  R_f x = \frac{1}{3}
\begin{bmatrix}2.1213 \\ 5.1213\end{bmatrix} \approx
\begin{bmatrix}0.7071 \\ 1.7071\end{bmatrix}.
\end{eqnarray*}
This is the reconstruction from the original $AR$-frame. Also,
\begin{eqnarray*}
S_f^{-1} E^* x = \frac{1}{3} \begin{bmatrix}2 & -1 \\ -1 &
2\end{bmatrix} \begin{bmatrix}-0.2 \\ 0\end{bmatrix} =
\begin{bmatrix}-0.1333
\\ 0.0667\end{bmatrix}.
\end{eqnarray*}
Therefore,
\begin{eqnarray*}
R_g x \approx R_f x - S_f^{-1} E^* x =
\begin{bmatrix}0.8404 \\ 1.6404\end{bmatrix}.
\end{eqnarray*}
So the first-order corrected reconstruction $R_g x$ is slightly
closer to the original signal $x = \begin{bmatrix}1
\\ 2\end{bmatrix}$ compared to $R_f x$, compensating for the
small perturbation in $f_1$.
\end{exam}

\section{\textbf{Acknowledgment}}
The author would like to thank the anonymous reviewers for their
comments and suggestions, which improved the presentation of the
results. Also, the author acknowledge the use of generative AI for
assistance in reviewing and improving this manuscript.
\bigskip
\section{\textbf{Statements and Declarations}}
\textbf{Conflict of interest statement.} The author declares that
there are no conflicts of interest.

\textbf{Data availability.} No datasets were generated or analyzed
during the current study.

\textbf{Funding.} The author declare that no funds, grants, or
other support were received during the preparation of this
manuscript.




\begin{thebibliography}{15}

\bibitem{Balazs.2010}
Balazs P, Antoine JP, Grybos A. weighted and controlled frames:
Mutual relationhsip and first numerical properties, Int. J.
Wavelets Multiresolution. Inf. Process, 2010;8:109-132.
doi:10.1142/50219691310003377

\bibitem{Balazs.2026}
Balazs P, Corso R, Stoeva D. Weighted frames, Weighted lower semi
frames and unconditionally convergent multipliers, J. Fourier
Anal. Appl. 2026:32:Paper NO.24. doi:10.1007/s00041-025-10216-0

\bibitem{Benedetto.1995}
Teolis A, Benedetto JJ. Local frames and noise reduction, Signal
Process, 1995;369-387. doi:10.1016/0165-1684(95)00064-K

\bibitem{Benedetto.1998}
Benedetto JJ. Noice reduction in terms of the theory of frames.
In: wavelet analysis and applications, Academic San Diego,
1998;7:259-284. doi:10.1016/S1874-608X(98)80010-1

\bibitem{Cai.2012}
Cai JF, Osher S, Shen Z. Image restoration: total variation,
wavelet frames, and byond, J. Am. Math. Soc. 2012;25:1033-1089.


\bibitem{Christensen.2023}
Christensen O. An Introduction to Frames and Riesz Bases.
Birkh\"user, 2003.

\bibitem{Duffin.1952}
Duffin RJ, and Schaeffer AC. A class of nonharmonic Fourier
series, Tran. AM. Math. Soc. 1952;72:341-366. doi:10.2307/1990760

\bibitem{Eldar.2012}
Eldar YC, G. Kutyniok G. Compressed Sensing: Theory and
Applications. Cambridge, 2012.

\bibitem{Heil.2011}
Heil C. A Basis Theory Primer. Birkh\"user, 2011.

\bibitem{Mallat.1999}
Mallat S. A wavelet tour of signal processing, Academic, San
Diego, 1999.

\bibitem{Meyer.1997}
Meyer FG, Coifman RR. Brushlets: a tool for directinal image
analysis and image compression, Appl. Comput. Harm. Anal.
1997;4(2):147-187. doi:10.1006/acha.1997.0208

\bibitem{Rajupillai.2019}
Rajupillai K, Palaniammal S. Frame reconstruction with noise
reduction in Hilbert space and applicatin in communication
systems, Math. Comput. Sim., 2019;155:324-334.
doi:10.1016/j.matcom.2018.06.014

\bibitem{Richard}
Varga RS. Matrix Iterative Analysis, 2nd ed., Springer, 2000

\bibitem{Strohmer.2001}
Strohmer T. Approximation of dual gaber frames, window decay and
wireless communications, Appl. Comput. Harm. Anal.
2001;11(2):243-262. doi:10.1006/acha.2001.0357

\bibitem{Yousef}
Saad Y. Iterative Methods for Sparse Linear Systems, 2nd ed.,
SIAM, 2003.

\end{thebibliography}
\bibliographystyle{plain}
\end{document}